\documentclass[reqno]{amsart}%
\usepackage{amssymb}
\usepackage{amsmath}
\usepackage{amsfonts}
\usepackage{graphicx}%
\providecommand{\U}[1]{\protect\rule{.1in}{.1in}}
\newtheorem{theorem}{Theorem}
\theoremstyle{plain}

\newtheorem{corollary}{Corollary}

\newtheorem{lemma}{Lemma}

\newtheorem{proposition}{Proposition}
\newtheorem{remark}{Remark}

\numberwithin{equation}{section}
\begin{document}
\title[All Invariant Regions and Global Solutions]{All Invariant Regions and Global Solutions for $m$-component
Reaction-Diffusion Systems with a Tridiagonal Symmetric Toeplitz Matrix of
Diffusion Coefficients }
\author{Salem ABDELMALEK}
\address{S. ABDELMALEK: Department of Mathematics, College of Sciences, Yanbu, Taibah
University, Saudi Arabia.}
\email{sallllm@gmail.com}

\begin{abstract}
The purpose of this paper is the construction of invariant regions in which we
establish the global existence of solutions for $m$-component
reaction-diffusion systems with a tridiagonal symmetric toeplitz matrix of
diffusion coefficients and with nonhomogeneous boundary conditions. The
proposed technique is based on invariant regions and Lyapunov functional
methods. The nonlinear reaction term has been supposed to be of polynomial growth.

\end{abstract}
\maketitle

\section{\textbf{Introduction}}

In recent years, the global existence for solutions for nonlinear parabolic
systems of partial differential equations have received considerable
attention. On of the most promising works that can be found in the litarature
is that of Jeff Morgan \cite{Morgan1}, where all the components satisfy the
same boundary conditions (Neumann or Dirichlet), and the reaction terms are
polynomially bounded and satisfy $m$ inequalities. Hollis later completed the
work of Morgan and established global existence in the presence of mixed
boundary conditions subject to certain structure requirements of the system.
In 2007, Abdelmalek and Kouachi \cite{Abdelmalek2} show that solutions of the
$m$-component reaction--diffusion systems with a diagonal diffusion matrix
exist globally (for $m\geq2$).

The results obtained in this work represent the proof of the global existence
of solutions with Neumann, Dirichlet, nonhomogeneous Robin and a mixture of
Dirichlet with nonhomogeneous Robin conditions. The reaction terms are again
assumed to be of polynomial growth and satisfy a single inequality. The
diffusion matrix is a simple symmetric tridiagonal one.

All along the paper, we will use the following notations and assumptions: we
denote by $m\geq2$ the number of equations of the system (i.e. an
$m$-component system):%
\begin{equation}
\left.
\begin{array}
[c]{l}%
\dfrac{\partial u_{1}}{\partial t}-a\Delta u_{1}-b\Delta u_{2}=f_{1}\left(
U\right)  ,\text{ \ \ \ \ \ \ \ \ \ \ \ }\\
\dfrac{\partial u_{\ell}}{\partial t}-b\Delta u_{\ell-1}-a\Delta u_{\ell
}-b\Delta u_{\ell+1}=f_{\ell}\left(  U\right)  ;\text{ }\ell=2,...,m-1\text{,
\ \ \ }\\
\dfrac{\partial u_{m}}{\partial t}-b\Delta u_{m-1}-a\Delta u_{m}=f_{m}\left(
U\right)  \text{, \ \ \ \ \ \ \ \ \ \ }%
\end{array}
\text{\ in }\Omega\times\left\{  t>0\right\}  ,\right.  \label{1.1}%
\end{equation}
with the boundary conditions:%
\begin{equation}
\alpha u_{\ell}+\left(  1-\alpha\right)  \partial_{\eta}u_{\ell}=\beta_{\ell
},\text{\ }\ell=1,...,m,\text{\ \ \ \ on }\partial\Omega\times\left\{
t>0\right\}  , \label{1.2}%
\end{equation}
and the initial data:%
\begin{equation}
u_{\ell}(0,x)=u_{\ell}^{0}(x),\text{\ }\ell=1,...,m,\text{ \ \ \ on}\;\Omega,
\label{1.3}%
\end{equation}
where:

\begin{enumerate}
\item[(i)] for nonhomogeneous Robin boundary conditions, we use\newline%
$0<\alpha<1$ $,$\ $\beta_{\ell}\in%
\mathbb{R}
,$ $\ell=1,...,m$, or

\item[(ii)] for homogeneous Neumann boundary conditions, we use\newline%
$\alpha=\beta_{\ell}=0,$ $\ell=1,...,m$, or

\item[(iii)] for homogeneous Dirichlet boundary conditions, we use\newline%
$1-\alpha=\beta_{\ell}=0,$ $\ell=1,...,m.$
\end{enumerate}

Here $\Omega$ is an open bounded domain of class $\mathbb{C}^{1}$ in $\mathbb{%
\mathbb{R}
}^{N}$ with boundary $\partial\Omega$, $\dfrac{\partial}{\partial\eta}$
denotes the outward normal derivative on $\partial\Omega$, and\ $U=\left(
u_{\ell}\right)  _{\ell=1}^{m}$.$\;$The constants $a$ and $b$ are supposed to
be positive non null and satisfying the condition:
\begin{equation}
2b\cos\frac{\pi}{m+1}<a. \label{1.4}%
\end{equation}
The initial data are assumed to be\ in the regions:%
\begin{equation}
\Sigma_{L,Z}=\left\{  \left(  u_{1}^{0},...,u_{m}^{0}\right)  \in%
\mathbb{R}
^{m}\text{ such that }\left\{
\begin{array}
[c]{l}%
\sum_{k=1}^{m}u_{k}^{0}\sin\frac{\left(  m+1-\ell\right)  k\pi}{m+1}\geq
0,\ell\in L\\
\sum_{k=1}^{m}u_{k}^{0}\sin\frac{zk\pi}{m+1}\leq0,z\in Z
\end{array}
\right.  \right\}  , \label{1.5}%
\end{equation}
with
\[
\left\{
\begin{array}
[c]{c}%
\sum_{k=1}^{m}\beta_{k}\sin\frac{\left(  m+1-\ell\right)  k\pi}{m+1}\geq
0,\ell\in L\\
\sum_{k=1}^{m}\beta_{k}\sin\frac{\left(  m+1-\ell\right)  k\pi}{m+1}\leq0,z\in
Z
\end{array}
\right.  ,
\]
where
\[
\left\{
\begin{array}
[c]{l}%
L\cap Z=\varnothing\\
L\cup Z=\left\{  1,2,...,m\right\}
\end{array}
\right.  .
\]
Hence, we can see that there are $2^{m}$ regions. The subsequent work is
similar for all of these regions as will be shown at the end of the paper. Let
us now examine the first region and then comment on the remaining cases. The
chosen region is the case where $L=\left\{  1,2,...,m\right\}  $ and
$Z=\varnothing$: we have%
\begin{equation}
\Sigma_{L,\varnothing}=\left\{  \left(  u_{1}^{0},...,u_{m}^{0}\right)  \in%
\mathbb{R}
^{m}\text{ such that }\sum_{k=1}^{m}u_{k}^{0}\sin\frac{\left(  m+1-\ell
\right)  k\pi}{m+1}\geq0,\ell\in L,\right\}  \label{1.6}%
\end{equation}
with%
\[
\sum_{k=1}^{m}\beta_{k}\sin\frac{\left(  m+1-\ell\right)  k\pi}{m+1}\geq
0,\ell\in L.
\]
The aim is now to study the global existence of slutions for the
reaction-diffusion system in (\ref{1.1}) in this region. In order to achieve
this aim, we need to diagonalize the diffusion matrix, see formula
(\ref{5.1}). First, let us define the reaction diffusion functions as:%
\begin{equation}
F_{\ell}\left(  w_{1},w_{2},...,w_{m}\right)  =\sum_{k=1}^{m}f_{k}\left(
U\right)  \sin\frac{\left(  m+1-\ell\right)  k\pi}{m+1}, \label{1.8}%
\end{equation}
where the variable $w_{\ell}$ is given by%
\begin{equation}
w_{\ell}=\sum_{k=1}^{m}u_{k}\sin\frac{\left(  m+1-\ell\right)  k\pi}{m+1}.
\label{1.7}%
\end{equation}
The defined function must satisfy the following three conditions:

\begin{enumerate}
\item[(A1)] The functions $F_{\ell}$ are continuously differentiable on $%
\mathbb{R}
_{+}^{m}$ for all $\ell=1,...,m$, satisfying $F_{\ell}(w_{1},...,w_{\ell
-1},0,w_{\ell+1},...,w_{m})\geq0$, for all $u_{\ell}$ $\geq0;$ $\ell=1,...,m$.
\end{enumerate}

\begin{enumerate}
\item[(A2)] The functions $F_{\ell}$ are of polynomial growth (see Hollis and
Morgan \cite{Hollis3}), which means that for all $\ell=1,...,m$ with integer
$N\geq1$:%
\begin{equation}
\left\vert F_{\ell}\left(  W\right)  \right\vert \leq C_{1}\left(
1+\overset{m}{\underset{\ell=1}{\sum}}w_{\ell}\right)  ^{N}\text{ on }\left(
0,+\infty\right)  ^{m}. \label{1.10}%
\end{equation}

\item[(A3)] The following inequality:%
\begin{equation}
\overset{m-1}{\underset{\ell=1}{\sum}}D_{\ell}F_{\ell}\left(  W\right)
+F_{m}\left(  W\right)  \leq C_{2}\left(  1+\overset{m}{\underset{\ell=1}%
{\sum}}w_{\ell}\right)  ; \label{1.11}%
\end{equation}
holds for all $w_{\ell}$ $\geq0;$ $\ell=1,...,m$ and all constants $D_{\ell
}\geq\overline{D_{\ell}};$ $\ell=1,...,m$ where $\overline{D_{\ell}};$
$\ell=1,...,m$ are positive constants sufficiently large. Note that $C_{1}$
and$C_{2}$ are positive and uniformly bounded functions defined on $%
\mathbb{R}
_{+}^{m}$.
\end{enumerate}

\section{Preliminary Observations and Notations}

The usual norms in spaces $L^{p}(\Omega)$, $L^{\infty}(\Omega)$ and
$C(\overline{\Omega})$ are denoted respectively by:%
\begin{align}
\left\Vert u\right\Vert _{p}^{p}  &  =\frac{1}{\left\vert \Omega\right\vert
}\int_{\Omega}\left\vert u(x)\right\vert ^{p}dx;\nonumber\\
\left\Vert u\right\Vert _{\infty}  &  =ess\underset{x\in\Omega}{\sup
}\left\vert u(x)\right\vert ,\label{3.1}\\
\left\Vert u\right\Vert _{C(\overline{\Omega})}  &  =\underset{x\in
\overline{\Omega}}{\max}\left\vert u(x)\right\vert .\nonumber
\end{align}

It is well-known that to prove the global existence of solutions to a
reaction-diffusion system (see Henry \cite{Henry}), it suffices to derive a
uniform estimate of the associated reaction term on $\left[  0;T_{\max
}\right)  $ in the space $L^{p}(\Omega)$ for some $p>n/2.$ Our aim is to
construct polynomial Lyapunov functionals allowing us to obtain $L^{p}-$
bounds on the components, which leads to global existence. Since the reaction
terms are continuously differentiable on $%
\mathbb{R}
_{+}^{m}$, it follows that for any initial data in $C(\overline{\Omega})$, it
is easy to check directly their Lipschitz continuity on bounded subsets of the
domain of a fractional power of the operator
\begin{equation}
O=-\left(
\begin{array}
[c]{cccc}%
\lambda_{1}\Delta & 0 & ... & 0\\
0 & \lambda_{2}\Delta & ... & 0\\
\vdots & \vdots & \ddots & \vdots\\
0 & 0 & ... & \lambda_{m}\Delta
\end{array}
\right)  . \label{3.2}%
\end{equation}
Under these assumptions, the following local existence result is well known
(see Friedman \cite{Friedman} and Pazy \cite{Pazy}).

\begin{remark}
Assumption\ (A1) contains smoothness and quasipositivity conditions that
guarantee local existence and nonnegativity of solutions as long as they
exist, via the maximum principle (see Smoller \cite{Smoller}). Assumption (A3)
is the usual polynomial growth condition necessary to obtain uniform bounds
from $p-$dependent $L^{P}$estimates. (see\ Abdelmalek and Kouachi
\cite{Abdelmalek2}, and Hollis and Morgan \cite{Hollis4}).
\end{remark}

\section{Some Properties of diffusion matrix}

\begin{lemma}
Considering the proposed reaction-diffusion system in (\ref{1.1}), the
resulting diffusion matrix can be given by:%
\begin{equation}
A=\left(
\begin{array}
[c]{cccccc}%
a & b & 0 & \cdots & 0 & 0\\
b & a & b & \ddots & 0 & 0\\
0 & b & a & \ddots & \vdots & \vdots\\
\vdots & \ddots & \ddots & \ddots & b & 0\\
0 & \cdots & 0 & b & a & b\\
0 & \cdots & 0 & 0 & b & a
\end{array}
\right)  . \label{2.1}%
\end{equation}
This matrix is said to be positive definite if the condition in (\ref{1.4}) is satisfied.
\end{lemma}

\begin{proof}
The proof of this lemma can be found in \cite{Johnson}. Note that if the
matrix is positive definite, it follows that $\det A>0.$
\end{proof}

\begin{lemma}
\label{Lemma0}The eigenvalues $\left(  \lambda_{\ell}<\lambda_{\ell-1};\text{
}\ell=2,...,m\right)  $ of $A$ are positive and\ are given by
\begin{equation}
\lambda_{\ell}=a+2b\cos\left(  \frac{\ell\pi}{m+1}\right)  , \label{2.2}%
\end{equation}
with the corresponding eigenvectors being\ $v_{\ell}=\left(  \sin\frac{\ell
\pi}{m+1},\sin\frac{2\ell\pi}{m+1},...,\sin\frac{m\ell\pi}{m+1}\right)  ^{t},$
for $\ell=1,...,m$. Hence, we conclude that $A$ is diagonalizable.\newline In
the remainder of this work we require an ascending order of the eigenvalues.
In order to simplify the indices in the formulas to come we define%
\begin{equation}
\bar{\lambda}_{\ell}=\lambda_{m+1-\ell}=a+2b\cos\left(  \frac{(m+1-\ell)\pi
}{m+1}\right)  ;\text{ }\ell=2,...,m, \label{2.3}%
\end{equation}
thus $\left(  \bar{\lambda}_{\ell}<\bar{\lambda}_{\ell+1};\text{ }%
\ell=2,...,m\right)  .$
\end{lemma}

\begin{proof}
Recall that the diffusion matrix is positive definite, hence it the
eigenvalues are necessarily positive. For\ an eigenpair $\left(
\lambda,X\right)  $, the components in $\left(  A-\lambda I\right)  X=0$ are%
\[
bx_{k-1}+\left(  a-\lambda\right)  x_{k}+bx_{k+1}=0,k=1,...,m,
\]
with $x_{0}=x_{m+1}=0,$ or equivalently,
\[
x_{k+2}+\left(  \frac{a-\lambda}{b}\right)  x_{k+1}+x_{k}=0,k=0,...,m-1.
\]

These are second-order homogeneous difference equations, and solving them is
similar to solving analogous differential equations. The technique is to seek
solutions of the form $x_{\ell}=\xi r^{k}$ for constants $\xi$ and $r.$ This
produces the quadratic equation
\[
r^{2}+\left(  \frac{a-\lambda}{b}\right)  r+1=0,
\]
with roots $r_{1}$ and $r_{2}.$ It can be argued that the general solution of
$x_{k+2}+\left(  \frac{a-\lambda}{b}\right)  x_{k+1}+x_{k}=0$ is%
\[
x_{\ell}=\left\{
\begin{array}
[c]{ll}%
\alpha r_{1}^{k}+\beta r^{k}, & \text{if }r_{1}\neq r_{2},\\
\alpha\rho^{k}+\beta k\rho^{k}, & \text{if \ }r_{1}=r_{2}=\rho,
\end{array}
\right.
\]
where $\alpha$ and $\beta$ are arbitrary constants.

For the eigenvalue problem at hand, $r_{1}$ and $r_{2}$ must be distinct
-otherwise $x_{k}=\alpha\rho^{k}+\beta k\rho^{k}$, and $x_{0}=x_{m+1}=0$
implies that each $x_{k}=0,$ which is impossible because $X$ is an
eigenvector. Hence, $x_{k}=\alpha r_{1}^{k}+\beta r^{k}$, and $x_{0}%
=x_{m+1}=0$ yields
\[
\left\{
\begin{array}
[c]{l}%
0=\alpha+\beta\\
0=\alpha r_{1}^{m+1}+\beta r_{2}^{m+1}%
\end{array}
\right.  \Rightarrow\left(  \frac{r_{1}}{r_{2}}\right)  ^{m+1}=\frac{-\beta
}{\alpha}=1\Rightarrow\frac{r_{1}}{r_{2}}=e^{\frac{2i\pi\ell}{m+1}},
\]
therefore, $r_{1}=r_{2}e^{\frac{2i\pi\ell}{m+1}}$ for some $1\leq\ell\leq m$.
This coupled with:
\[
r^{2}+\left(  \frac{a-\lambda}{b}\right)  r+1=\left(  r-r_{1}\right)  \left(
r-r_{2}\right)  \Rightarrow\left\{
\begin{array}
[c]{l}%
r_{1}r_{2}=1\\
r_{1}+r_{2}=-\left(  \frac{a-\lambda}{b}\right)
\end{array}
\right.  ,
\]
leads to $r_{1}=e^{\frac{i\pi\ell}{m+1}}$, $r_{2}=e^{-\frac{i\pi\ell}{m+1}}$,
and
\[
\lambda=a+b\left(  e^{\frac{i\pi\ell}{m+1}}+e^{-\frac{i\pi\ell}{m+1}}\right)
=a+2b\cos\left(  \frac{\ell\pi}{m+1}\right)  \text{.}%
\]
The eigenvalues of $A$ can, therefore, be given by
\[
\lambda_{\ell}=a+2b\cos\left(  \frac{\ell\pi}{m+1}\right)  ;
\]
for $\ell=1,...,m.$

Since these $\lambda_{\ell}$'s are all distinct ($\cos\theta$ is a strictly
decreasing function of $\theta$\ on $\left(  0,\pi\right)  ,$ and $b\neq0$),
$A$ is necessarily diagonalizable.

Finally, the $\ell^{th}$ component of any eigenvector associated with
$\lambda_{\ell}$ satisfies $x_{k}=\alpha r_{1}^{k}+\beta r_{2}^{k}$ with
$\alpha+\beta=0$, thus%
\[
x_{k}=\alpha\left(  e^{\frac{2i\pi k}{m+1}}-e^{-\frac{2i\pi k}{m+1}}\right)
=2i\alpha\sin\left(  \frac{k}{m+1}\pi\right)  .
\]
Setting $\alpha=\frac{1}{2i}$ yields a particular eignvector associated with
$\lambda_{\ell}$ given by
\[
v_{\ell}=\left(  \sin\left(  \frac{1\ell\pi}{m+1}\right)  ,\sin\left(
\frac{2\ell\pi}{m+1}\right)  ,...,\sin\left(  \frac{m\ell\pi}{m+1}\right)
\right)  ^{t}.
\]
Because the $\lambda_{\ell}$'s are distinct, $\left\{  v_{1},v_{2}%
,...,v_{m}\right\}  ,$ is a compleat linearly independent set, so $\left(
v_{1}\shortmid v_{2}\shortmid...\shortmid v_{m}\right)  $
\textbf{diagonalizes} $A$.

Now, let us prove that
\[
\lambda_{\ell}<\lambda_{\ell-1};\text{ }\ell=2,...,m.
\]
We have%
\[
\ell>\ell-1.
\]

Dividing by $\left(  m+1\right)  $ and multiplying by $\pi,$ we have%
\[
\frac{\ell\pi}{m+1}>\frac{\left(  \ell-1\right)  \pi}{m+1},
\]
The cosine function $\cos\theta$ is strictly decreasing in $\theta$\ on
$\left(  0,\pi\right)  ,$ thus we have%
\[
\cos\left(  \frac{\ell\pi}{m+1}\right)  <\cos\left(  \frac{\left(
\ell-1\right)  \pi}{m+1}\right)  .
\]

Finally, multiplying both sides of the inequality by $2b$ and adding $a$
yields%
\[
\lambda_{\ell}=a+2b\cos\left(  \frac{\ell\pi}{m+1}\right)  <a+2b\cos\left(
\frac{\left(  \ell-1\right)  \pi}{m+1}\right)  =\lambda_{\ell-1}.
\]

\end{proof}

\section{Main Results}

\begin{proposition}
\label{proposition1}The eigenvectors of the diffusion matrix associated with
the eigenvalues $\overline{\lambda}_{\ell}$ are defined as $\overline{v}%
_{\ell}=\left(  \overline{v}_{\ell1},\overline{v}_{\ell2},...,\overline
{v}_{\ell m}\right)  ^{t}$. Multiplying\ each of the $m$ equations in
(\ref{1.1}) by the corresponding element of the $\ell^{\text{th}}$ eigenvector
and adding the equations together yields:%
\begin{equation}
\dfrac{\partial w_{\ell}}{\partial t}-\bar{\lambda}_{\ell}\Delta w_{\ell
}=F_{\ell}\left(  w_{1},w_{2},...,w_{m}\right)  , \label{5.1}%
\end{equation}
and%
\begin{equation}
\alpha w_{\ell}+\left(  1-\alpha\right)  \partial_{\eta}w_{\ell}=\rho_{\ell
}\text{\ \ \ \ \ on }\partial\Omega\times\left\{  t>0\right\}  . \label{5.2}%
\end{equation}
The reaction term $F_{\ell}$, the components $w_{\ell}$, and the ascending
order eigenvalues $\bar{\lambda}_{\ell}$ have been defined perviously in this paper.
\end{proposition}

Note that condition (\ref{1.4}) guarantees the parabolicity of the proposed
reaction-diffusion system in (\ref{1.1})-(\ref{1.3}), which implies that the
system described by (\ref{5.1})-(\ref{5.2}) is equivalent to it in the region:%

\[
\Sigma_{L,\varnothing}=\left\{  \left(  u_{1}^{0},...,u_{m}^{0}\right)  \in%
\mathbb{R}
^{m}\text{ such that }\left\{  w_{\ell}^{0}=\sum_{k=1}^{m}u_{k}^{0}\sin
\frac{\left(  m+1-\ell\right)  k\pi}{m+1}\geq0,\ell\in L\right.  ,\right.
\]
with%
\[
\left\{  \rho_{\ell}^{0}=\sum_{k=1}^{m}\beta_{k}\sin\frac{\left(
m+1-\ell\right)  k\pi}{m+1}\geq0,\ell\in L\right.  .
\]
This implies that the components $w_{\ell}$ are necessarily positive.

\begin{proposition}
The system (\ref{5.1})-(\ref{5.2}) admits a unique classical solution
$(w_{1};w_{2};..,w_{m})$ on$\ (0,T_{\max})\times\Omega$.%
\begin{equation}
\text{If }T_{\max}<\infty\text{ then }\underset{t\nearrow T_{\max}}{\lim
}\overset{m}{\underset{\ell=1}{\sum}}\left\Vert w_{\ell}\left(  t,.\right)
\right\Vert _{\infty}=\infty\text{,} \label{3.3}%
\end{equation}
\newline where $T_{\max}$ $\left(  \left\Vert w_{1}^{0}\right\Vert _{\infty
},\left\Vert w_{2}^{0}\right\Vert _{\infty},...,\left\Vert w_{m}%
^{0}\right\Vert _{\infty}\right)  $ denotes the eventual blow-up time.
\end{proposition}

The main result of the paper reads as follows.

\begin{theorem}
\label{theorem1}Suppose that the functions $F_{\ell};$ $\ell=1,...,m$ are of
polynomial growth and satisfy condition (\ref{1.11}) for some positive
constants $D_{\ell};$ $\ell=1,...,m$ sufficiently large. Let $\left(
w_{1}\left(  t,.\right)  ,w_{2}\left(  t,.\right)  ,...,w_{m}\left(
t,.\right)  \right)  $ be a solution of (\ref{5.1})-(\ref{5.2}) and%
\begin{equation}
L(t)=\int_{\Omega}H_{p_{m}}\left(  w_{1}\left(  t,x\right)  ,w_{2}\left(
t,x\right)  ,...,w_{m}\left(  t,x\right)  \right)  dx, \label{5.3}%
\end{equation}
where%
\[
H_{p_{m}}\left(  w_{1},...,w_{m}\right)  =\overset{p_{m}}{\underset{p_{m-1}%
=0}{\sum}}...\overset{p_{2}}{\underset{p_{1}=0}{\sum}}C_{p_{m}}^{p_{m-1}%
}...C_{p_{2}}^{p_{1}}\theta_{1}^{p_{1}^{2}}...\theta_{\left(  m-1\right)
}^{p_{\left(  m-1\right)  }^{2}}w_{1}^{p_{1}}w_{2}^{p_{2}-p_{1}}%
...w_{m}^{p_{m}-p_{m-1}},
\]
with $p_{m}$ a positive integer and $C_{p_{\kappa}}^{p_{\ell}}=\frac
{p_{\kappa}!}{p_{\ell}!\left(  p_{\kappa}-p_{\ell}\right)  !}$.\newline Also
suppose that the following condition is satisfied%
\begin{equation}
K_{l}^{l}>0;\text{ }l=2,...,m\text{,} \label{1.12}%
\end{equation}
where%
\[
K_{l}^{r}=K_{r-1}^{r-1}\times K_{l}^{r-1}-\left[  H_{l}^{r-1}\right]
^{2};r=3,...,l\text{,}%
\]%
\[
H_{l}^{r}=\underset{1\leq\ell,\kappa\leq l}{\det}\left(  \left(
a_{\ell,\kappa}\right)  _{\substack{\ell\neq l,...r+1\\\kappa\neq
l-1,..r}}\right)  \times\overset{k=r-2}{\underset{k=1}{\Pi}}\left(
\det\left[  k\right]  \right)  ^{2^{\left(  r-k-2\right)  }}%
r=3,...,l-1\text{,}%
\]%
\[
K_{l}^{2}=\underset{\text{positive value}}{\underbrace{\bar{\lambda}_{1}%
\bar{\lambda}_{l}\overset{l-1}{\underset{k=1}{\Pi}}\theta_{k}^{2\left(
p_{k}+1\right)  ^{2}}\times\overset{m-1}{\underset{k=l}{\Pi}}\theta
_{k}^{2\left(  p_{k}+2\right)  ^{2}}}}\times\left[  \overset{l-1}%
{\underset{k=1}{\Pi}}\theta_{k}^{2}-A_{1l}^{2}\right]  ,
\]
and%
\[
H_{l}^{2}=\underset{\text{positive value}}{\underbrace{\bar{\lambda}_{1}%
\sqrt{\bar{\lambda}_{2}\bar{\lambda}_{l}}\theta_{1}^{2\left(  p_{1}+1\right)
^{2}}\overset{l-1}{\underset{k=2}{\Pi}}\theta_{k}^{\left(  p_{k}+2\right)
^{2}+\left(  p_{k}+1\right)  ^{2}}\times\overset{m-1}{\underset{k=l}{\Pi}%
}\theta_{k}^{2\left(  p_{k}+2\right)  ^{2}}}}\times\left[  \theta_{1}%
^{2}A_{2l}-A_{12}A_{1l}\right]  .
\]
$\underset{1\leq\ell,\kappa\leq l}{\det}\left(  \left(  a_{\ell,\kappa
}\right)  _{\substack{\ell\neq l,...r+1\\\kappa\neq l-1,..r}}\right)  $ is
denoted determinant of $r$ square symmetric matrix obtained from $\left(
a_{\ell,\kappa}\right)  _{1\leq\ell,\kappa\leq m}$ by removing the $\left(
r+1\right)  ^{th},\left(  r+2\right)  ^{th},...,l^{th}$ rows and the
$r^{th},\left(  r+1\right)  ^{th},...,\left(  l-1\right)  ^{th}$ columns.
where $\det\left[  1\right]  ,...,\det\left[  m\right]  $ \ are the minors of
the matrix $\left(  a_{\ell,\kappa}\right)  _{1\leq\ell,\kappa\leq m}.$ The
elements of the matrix are:%
\begin{equation}
a_{\ell\kappa}=\frac{\bar{\lambda}_{\ell}+\bar{\lambda}_{\kappa}}{2}\theta
_{1}^{p_{1}^{2}}...\theta_{\left(  \ell-1\right)  }^{p_{\left(  \ell-1\right)
}^{2}}\theta_{\ell}^{\left(  p_{\ell}+1\right)  ^{2}}...\theta_{\kappa
-1}^{\left(  p_{\left(  \kappa-1\right)  }+1\right)  ^{2}}\theta_{\kappa
}^{\left(  p_{\kappa}+2\right)  ^{2}}...\theta_{\left(  m-1\right)  }^{\left(
p_{\left(  m-1\right)  }+2\right)  ^{2}}. \label{1.13}%
\end{equation}
where $\bar{\lambda}_{\ell}$ in (\ref{2.2})- (\ref{2.3}). Note that
$A_{\ell\kappa}=\dfrac{\bar{\lambda}_{\ell}+\bar{\lambda}_{\kappa}}%
{2\sqrt{\bar{\lambda}_{\ell}\bar{\lambda}_{\kappa}}}$ for all $\ell
,\kappa=1,...,m$. and $\theta_{\ell};\ell=1,...,\left(  m-1\right)  $ are
positive constants.\newline It follows from these conditions that the
functional $L$ is uniformly bounded on the interval $\left[  0,T^{\ast
}\right]  ,T^{\ast}<T_{\max}$.
\end{theorem}

\begin{corollary}
\label{corollary1}Under the assumptions of theorem \ref{theorem1}, all
solutions of (\ref{5.1})-(\ref{5.2}) with positive initial data in $L^{\infty
}\left(  \Omega\right)  $ are in $L^{\infty}\left(  0,T^{\ast};L^{p}\left(
\Omega\right)  \right)  $ for some $p\geq1.$
\end{corollary}

\begin{proposition}
\label{proposition2}Under the assumptions of theorem \ref{theorem1} and given
that the condition (\ref{1.4}) is satisfied, all solutions of (\ref{5.1}%
)-(\ref{5.2}) with positive initial data in $L^{\infty}\left(  \Omega\right)
$ are global for some $p>\dfrac{Nn}{2}$.
\end{proposition}

\section{\textbf{Proofs}}

For the proof of theorem \ref{theorem1}, we first need to define some
preparatory Lemmas.

\begin{lemma}
\label{Lemma1}With\ $H_{p_{m}}$ being the homogeneous polynomial defined by
(\ref{5.3}), differentiating in $w_{1}$ yields%
\begin{align}
\partial_{w_{1}}H_{p_{m}}  &  =p_{m}\overset{p_{m}-1}{\underset{p_{m-1}%
=0}{\sum}}..\overset{p_{2}}{\underset{p_{1}=0}{\sum}}C_{p_{m}-1}^{p_{m-1}%
}...C_{p_{2}}^{p_{1}}\theta_{1}^{\left(  p_{1}+1\right)  ^{2}}...\theta
_{\left(  m-1\right)  }^{\left(  p_{\left(  m-1\right)  }+1\right)  ^{2}%
}\times\nonumber\\
&  w_{1}^{p_{1}}w_{2}^{p_{2}-p_{1}}w_{3}^{p_{3}-p_{2}}...w_{m}^{\left(
p_{m}-1\right)  -p_{m-1}}. \label{5.4}%
\end{align}
\newline Similarly for $\ell=2,...,m-1$, we have%
\begin{align}
\partial_{w_{\ell}}H_{p_{m}}  &  =p_{m}\overset{p_{m}-1}{\underset{p_{m-1}%
=0}{\sum}}...\overset{p_{2}}{\underset{p_{1}=0}{\sum}}C_{p_{m}-1}^{p_{m-1}%
}...C_{p_{2}}^{p_{1}}\theta_{1}^{p_{1}^{2}}...\theta_{\ell-1}^{p_{\left(
\ell-1\right)  }^{2}}\theta_{\ell}^{\left(  p_{\ell}+1\right)  2}%
...\theta_{\left(  m-1\right)  }^{\left(  p_{\left(  m-1\right)  }+1\right)
^{2}}\times\nonumber\\
&  w_{1}^{p_{1}}w_{2}^{p_{2}-p_{1}}w_{3}^{p_{3}-p_{2}}...w_{m}^{\left(
p_{m}-1\right)  -p_{m-1}}. \label{5.5}%
\end{align}
Finally, differentiating in $w_{m}$ yields%
\begin{align}
\partial_{w_{m}}H_{p_{m}}  &  =p_{m}\overset{p_{m}-1}{\underset{p_{m-1}%
=0}{\sum}}...\overset{p_{2}}{\underset{p_{1}=0}{\sum}}C_{p_{m}-1}^{p_{m-1}%
}...C_{p_{3}}^{p_{2}}C_{p_{2}}^{p_{1}}\theta_{1}^{p_{1}^{2}}\theta_{2}%
^{p_{2}^{2}}...\theta_{\left(  m-1\right)  }^{p_{\left(  m-1\right)  }^{2}%
}\times\nonumber\\
&  w_{1}^{p_{1}}w_{2}^{p_{2}-p_{1}}w_{3}^{p_{3}-p_{2}}...w_{m}^{\left(
p_{m}-1\right)  -p_{m-1}}. \label{5.6}%
\end{align}

\end{lemma}

\begin{lemma}
\label{Lemma2}The second partial derivative of $H_{p_{m}}$ in $w_{1}$\ is
given by%
\begin{align}
\partial_{w_{1}^{2}}H_{n}  &  =p_{m}\left(  p_{m}-1\right)  \overset{p_{m}%
-2}{\underset{p_{m-1}=0}{\sum}}...\overset{p_{3}}{\underset{p_{2}=0}{\sum}%
}\overset{p_{2}}{\underset{p_{1}=0}{\sum}}C_{p_{m}-2}^{p_{m-1}}...C_{p_{2}%
}^{p_{1}}\times\nonumber\\
&  \theta_{1}^{\left(  p_{1}+2\right)  ^{2}}...\theta_{\left(  m-1\right)
}^{\left(  p_{\left(  m-1\right)  }+2\right)  ^{2}}w_{1}^{p_{1}}w_{2}%
^{p_{2}-p_{1}}...w_{m}^{\left(  p_{m}-2\right)  -p_{m-1}}. \label{5.7}%
\end{align}
Similarly, we obtain%
\begin{align}
\partial_{w_{\ell}^{2}}H_{n}  &  =p_{m}\left(  p_{m}-1\right)  \overset
{p_{m}-2}{\underset{p_{m-1}=0}{\sum}}...\overset{p_{2}}{\underset{p_{1}%
=0}{\sum}}C_{p_{m}-2}^{p_{m-1}}...C_{p_{2}}^{p_{1}}\times\nonumber\\
&  \theta_{1}^{p_{1}^{2}}\theta_{2}^{p_{2}^{2}}...\theta_{\ell-1}^{p_{\ell
-1}^{2}}\theta_{\ell}^{\left(  p_{\ell}+2\right)  ^{2}}...\theta_{\left(
m-1\right)  }^{\left(  p_{\left(  m-1\right)  }+2\right)  ^{2}}\times
\nonumber\\
&  w_{1}^{p_{1}}w_{2}^{p_{2}-p_{1}}...w_{m}^{\left(  p_{m}-2\right)  -p_{m-1}%
}. \label{5.8}%
\end{align}
for all\ $\ell=2,...,m-1$,
\begin{align}
\partial_{w_{\ell}w_{\kappa}}H_{n}  &  =p_{m}\left(  p_{m}-1\right)
\overset{p_{m}-2}{\underset{p_{m-1}=0}{\sum}}...\overset{p_{2}}{\underset
{p_{1}=0}{\sum}}C_{p_{m}-2}^{p_{m-1}}...C_{p_{2}}^{p_{1}}\times\nonumber\\
&  \theta_{1}^{p_{1}^{2}}...\theta_{\ell-1}^{p_{\ell-1}^{2}}\theta_{\ell
}^{\left(  p_{\ell}+1\right)  ^{2}}...\theta_{\kappa-1}^{\left(  p_{\kappa
-1}+1\right)  ^{2}}\theta_{\kappa}^{\left(  p_{\kappa}+2\right)  ^{2}%
}...\theta_{\left(  m-1\right)  }^{\left(  p_{\left(  m-1\right)  }+2\right)
^{2}}\times\nonumber\\
&  w_{1}^{p_{1}}w_{2}^{p_{2}-p_{1}}...w_{m}^{\left(  p_{m}-2\right)  -p_{m-1}}
\label{5.9}%
\end{align}
for all\ $1\leq\ell<\kappa\leq m$. Finally, the second derivative in $w_{m}$
is given by%
\begin{align}
\partial_{w_{m}^{2}}H_{n}  &  =p_{m}\left(  p_{m}-1\right)  \overset{p_{m}%
-2}{\underset{p_{m-1}=0}{\sum}}...\overset{p_{2}}{\underset{p_{1}=0}{\sum}%
}C_{p_{m}-2}^{p_{m-1}}...C_{p_{2}}^{p_{1}}\theta_{1}^{p_{1}^{2}}%
...\theta_{\left(  m-1\right)  }^{p_{\left(  m-1\right)  }^{2}}\times
\nonumber\\
&  w_{1}^{p_{1}}w_{2}^{p_{2}-p_{1}}...w_{m}^{\left(  p_{m}-2\right)  -p_{m-1}%
}. \label{5.10}%
\end{align}

\end{lemma}

\begin{lemma}
[see Abdelmalek and Kouachi \cite{Abdelmalek2}]\label{Lemma3}Let $A$ be
the\ $m$-square symetric matrix defined by $A=\left(  a_{\ell\kappa}\right)
_{1\leq\ell,\kappa\leq m}$ then we the following property arises:%
\begin{equation}
\left\{
\begin{array}
[c]{l}%
K_{m}^{m}=\det\left[  m\right]  \times\overset{k=m-2}{\underset{k=1}{\Pi}%
}\left(  \det\left[  k\right]  \right)  ^{2^{\left(  m-k-2\right)  }},\text{
\ \ }m>2\\
K_{2}^{2}=\det\left[  2\right]
\end{array}
\right.  \label{5.11}%
\end{equation}
where%
\begin{align*}
K_{m}^{l}  &  =K_{l-1}^{l-1}\times K_{m}^{l-1}-\left(  H_{m}^{l-1}\right)
^{2};l=3,...,m,\\
H_{m}^{l}  &  =\underset{1\leq\ell,\kappa\leq m}{\det}\left(  \left(
a_{\ell,\kappa}\right)  _{\substack{\ell\neq m,...l+1\\\kappa\neq
m-1,..l}}\right)  \times\overset{k=l-2}{\underset{k=1}{\Pi}}\left(
\det\left[  k\right]  \right)  ^{2^{\left(  l-k-2\right)  }}l=3,...,m-1,\\
K_{m}^{2}  &  =a_{11}a_{mm}-\left(  a_{1m}\right)  ^{2},\\
H_{m}^{2}  &  =a_{11}a_{2m}-a_{12}a_{1m}.
\end{align*}

\end{lemma}

\begin{proof}
[Proof of Theorem \ref{theorem1}]The aim of this work is to prove that $L(t)$
is uniformly bounded on the interval $\left[  0,T^{\ast}\right]  ,T^{\ast
}<T_{\max}$. Let us start by differentiating $L$ with respect to $t$:%
\begin{align*}
L^{\prime}(t)  &  =\int_{\Omega}\partial_{t}H_{p_{m}}dx\\
&  =\int_{\Omega}\overset{m}{\underset{\ell=1}{\sum}}\partial_{w_{\ell}%
}H_{p_{m}}\frac{\partial w_{\ell}}{\partial t}dx\\
&  =\int_{\Omega}\overset{m}{\underset{\ell=1}{\sum}}\partial_{w_{\ell}%
}H_{p_{m}}\left(  \lambda_{\ell}\Delta w_{\ell}+F_{\ell}\right)  dx\\
&  =\int_{\Omega}\overset{m}{\underset{\ell=1}{\sum}}\bar{\lambda}_{\ell
}\partial_{w_{\ell}}H_{p_{m}}\Delta w_{\ell}dx+\int_{\Omega}\overset
{m}{\underset{\ell=1}{\sum}}\partial_{w_{\ell}}H_{p_{m}}F_{\ell}dx\\
&  =I+J,
\end{align*}

where%
\begin{equation}
I=\int_{\Omega}\overset{m}{\underset{\ell=1}{\sum}}\bar{\lambda}_{\ell
}\partial_{w_{\ell}}H_{p_{m}}\Delta w_{\ell}dx, \label{5.11a}%
\end{equation}
and%
\begin{equation}
J=\int_{\Omega}\overset{m}{\underset{\ell=1}{%
{\displaystyle\sum}
}}\partial_{w_{\ell}}H_{p_{m}}F_{\ell}dx. \label{5.12}%
\end{equation}

Using Green's formula, we can divide $I$ into two parts $I_{1}$ and $I_{2}$
where%
\begin{equation}
I_{1}=\int_{\partial\Omega}\overset{m}{\underset{\ell=1}{\sum}}\bar{\lambda
}_{\ell}\partial_{w_{\ell}}H_{p_{m}}\partial_{\eta}w_{\ell}dx, \label{5.13}%
\end{equation}
and
\begin{equation}
I_{2}=-\int_{\Omega}\left[  \left(  \left(  \frac{\overline{\lambda}_{\ell
}+\overline{\lambda}_{\kappa}}{2}\partial_{w_{\kappa}w_{\ell}}H_{p_{m}%
}\right)  _{1\leq\ell,\kappa\leq m}\right)  T\right]  \cdot Tdx \label{5.14}%
\end{equation}
for $p_{1}=0,...,p_{2},$ $p_{2}=0,...,p_{3}$ $...p_{m-1}=0,...,p_{m}-2$ and
$T=\left(  \nabla w_{1},\nabla w_{2},...,\nabla w_{m}\right)  ^{t}.$ Applying
lemmas \ref{Lemma1} and \ref{Lemma2} yields%
\begin{equation}
\left.
\begin{array}
[c]{l}%
\left(  \frac{\bar{\lambda}_{\ell}+\bar{\lambda}_{\kappa}}{2}\partial
_{w_{\kappa}w_{\ell}}H_{p_{m}}\right)  _{1\leq\ell,\kappa\leq m}=\\
p_{m}\left(  p_{m}-1\right)  \overset{p_{m}-2}{\underset{p_{m-1}=0}{\sum}%
}...\overset{p_{2}}{\underset{p_{1}=0}{\sum}}C_{p_{m}-2}^{p_{m-1}}...C_{p_{2}%
}^{p_{1}}\left(  \left(  a_{\ell\kappa}\right)  _{1\leq\ell,\kappa\leq
m}\right)  w_{1}^{p_{1}}...w_{m}^{\left(  p_{m}-2\right)  -p_{m-1}},
\end{array}
\right.  \label{5.15}%
\end{equation}
where $\left(  a_{\ell\kappa}\right)  _{1\leq\ell,\kappa\leq m}$ is the matrix
defined in formula (\ref{1.13}).

Now the proof of positivity for $I$ simplifies to proving that there exists a
positive constant $C_{4}$ independent of $t\in\left[  0,T_{\max}\right)  $
such that%
\begin{equation}
I_{1}\leq C_{4}\text{ for all }t\in\left[  0,T_{\max}\right)  , \label{5.16}%
\end{equation}
and that%
\begin{equation}
I_{2}\leq0, \label{5.17}%
\end{equation}
for several boundary conditions. First, let us prove the formula in
(\ref{5.16}):

(i) If $\ell=1,...m$ $:0<\alpha_{\ell}<1$ , then using the boundary conditions
(\ref{1.2}) we get
\[
I_{1}=\int_{\partial\Omega}\overset{m}{\underset{\ell=1}{\sum}}\bar{\lambda
}_{\ell}\partial_{w_{\ell}}H_{p_{m}}\left(  \gamma_{\ell}-\sigma_{\ell}%
w_{\ell}\right)  dx,
\]
where $\sigma_{\ell}=\dfrac{\alpha_{\ell}}{1-\alpha_{\ell}}$ and $\gamma
_{\ell}=\dfrac{\beta_{\ell}}{1-\alpha_{\ell}}$, for\ $\ell=1,...m$. Since
$H\left(  W\right)  =\overset{m}{\underset{\ell=1}{\sum}}\lambda_{\ell
}\partial_{w_{\ell}}H_{p_{m}}\left(  \gamma_{\ell}-\sigma_{\ell}w_{\ell
}\right)  =P_{n-1}\left(  W\right)  -Q_{n}\left(  W\right)  $, where $P_{n-1}$
and $Q_{n}$ are polynomials with positive coefficients and respective degrees
$n-1$ and $n$, and since the solution is positive, it follows that
\begin{equation}
\underset{\overset{m}{\underset{\ell=1}{\sum}}\left\vert w_{\ell}\right\vert
\rightarrow+\infty}{\lim\sup}H\left(  W\right)  =-\infty, \label{5.18}%
\end{equation}
which proves that $H$ is uniformly bounded on $%
\mathbb{R}
_{+}^{m}$\ and consequently (\ref{5.16}).

(ii)$\;$If for all $\ell=1,...m:\alpha=0$, then $I_{1}=0$ on $\left[
0,T_{\max}\right)  $.

(iii) The case of homogeneous Dirichlet conditions is trivial, since in this
case the positivity of the solution on $\left[  0,T_{\max}\right)
\times\Omega$ implies $\partial_{\eta}w_{\ell}\leq0,\forall\ell=1,...m$ on
$\left[  0,T_{\max}\right)  \times\partial\Omega$. Consequently, one gets
obtains the same result in (\ref{5.16}) with $C_{4}=0$.

Hence, the proof of (\ref{5.16}) is complete. Now, we move to the proof of
(\ref{5.17}):

Recall the matrix $\left(  a_{\ell\kappa}\right)  _{1\leq\ell,\kappa\leq m}$
which was defined in formula (\ref{1.13}). The quadratic forms (with respect
to $\nabla w_{\ell},$ $\ell=1,...,m$) associated with the matrix $\left(
a_{\ell\kappa}\right)  _{1\leq\ell,\kappa\leq m}$, with $p_{1}=0,...,p_{2},$
$p_{2}=0,..,p_{3}$ ... $p_{m-1}=0,...,p_{m}-2$, is positive definite since its
minors $\det\left[  1\right]  $, $\det\left[  2\right]  $,$...$ $\det\left[
m\right]  $ are all positive. Let us examine these minors and prove their
positivity by induction:

The first minor
\[
\det\left[  1\right]  =\lambda_{1}\theta_{1}^{\left(  p_{1}+2\right)  ^{2}%
}\theta_{2}^{\left(  p_{2}+2\right)  ^{2}}..\theta_{\left(  m-1\right)
}^{\left(  p_{\left(  m-1\right)  }+2\right)  ^{2}}>0
\]
is trivial for $p_{1}=0,...,p_{2},$ $p_{2}=0,...,p_{3}$ ... $p_{m-1}%
=0,...,p_{m}-2$.

For the second minor $\det\left[  2\right]  $,\ according to lemma
\ref{Lemma3}, we get:%
\[
\det\left[  2\right]  =K_{2}^{2}=\lambda_{1}\lambda_{2}\theta_{1}^{2\left(
p_{1}+1\right)  ^{2}}\overset{m-1}{\underset{k=2}{\Pi}}\theta_{k}^{2\left(
p_{k}+2\right)  ^{2}}\left[  \theta_{1}^{2}-A_{12}^{2}\right]  .
\]

Using (\ref{1.12})\ for $l=2$\ we get $\det\left[  2\right]  >0.$

Similarly, for the third minor $\det\left[  3\right]  $,\ and again using
lemma \ref{Lemma3}, we have:%
\[
K_{3}^{3}=\det\left[  3\right]  \det\left[  1\right]  .
\]
\newline Since $\det1>0$, we conclude that%
\[
\operatorname*{sign}(K_{3}^{3})=\operatorname*{sign}(\det\left[  3\right]  ).
\]

Again, using (\ref{1.12})\ for $l=3$ yields $\det\left[  3\right]  >0$.

To conclude the proof, let us suppose $\det\left[  k\right]  >0$ for
$k=1,2,...,l-1$ and show that $\det[l]$ is necessarily positive. We have%
\begin{equation}
\det\left[  k\right]  >0,k=1,...,\left(  l-1\right)  \Rightarrow
\overset{k=l-2}{\underset{k=1}{\Pi}}\left(  \det\left[  k\right]  \right)
^{2^{\left(  l-k-2\right)  }}>0. \label{5.20}%
\end{equation}
From lemma \ref{Lemma3}, we obtain $K_{l}^{l}=\det\left[  l\right]
\times\overset{k=l-2}{\underset{k=1}{\Pi}}\left(  \det\left[  k\right]
\right)  ^{2^{\left(  l-k-2\right)  }}$, and from$\ $(\ref{5.20}), we get:
$\operatorname*{sign}(K_{l}^{l})=\operatorname*{sign}\left(  \det\left[
l\right]  \right)  $. Since $K_{l}^{l}>0$ according to (\ref{1.12}) then
$\det\left[  l\right]  >0$ and the proof of (\ref{5.17}) is concluded. It
follows from (\ref{5.16}) and (\ref{5.17}) that $I$ is bounded. Now, let us
divert our attention to proving that $J$ in (\ref{5.12}) is bounded.
Substituting the expressions of the partial derivatives given by \ref{Lemma1}
in the second integral of (\ref{5.12}) yields%
\begin{align*}
J  &  =\int_{\Omega}\left[  p_{m}\overset{p_{m}-1}{\underset{p_{m-1}=0}{\sum}%
}...\overset{p_{2}}{\underset{p_{1}=0}{\sum}}C_{p_{m}-1}^{p_{m-1}}...C_{p_{2}%
}^{p_{1}}w_{1}^{p_{1}}w_{2}^{p_{2}-p_{1}}...w_{m}^{p_{m}-1-p_{m-1}}\right]
\times\\
&  \left(  \overset{m-1}{\underset{\ell=1}{\Pi}}\theta_{\ell}^{\left(
p_{\ell}+1\right)  ^{2}}F_{1}+\overset{m-1}{\underset{\kappa=2}{\sum}}%
\overset{\kappa-1}{\underset{k=1}{\Pi}}\theta_{k}^{p_{k}^{2}}\overset
{m-1}{\underset{\ell=\kappa}{\Pi}}\theta_{\ell}^{\left(  p_{\ell}+1\right)
^{2}}F_{\kappa}+\overset{m-1}{\underset{\ell=1}{\Pi}}\theta_{\ell}^{p_{\ell
}^{2}}F_{m}\right)  dx\\
&  =\int_{\Omega}\left[  p_{m}\overset{p_{m}-1}{\underset{p_{m-1}=0}{\sum}%
}...\overset{p_{2}}{\underset{p_{1}=0}{\sum}}C_{p_{m}-1}^{p_{m-1}}...C_{p_{2}%
}^{p_{1}}w_{1}^{p_{1}}w_{2}^{p_{2}-p_{1}}...w_{m}^{p_{m}-1-p_{m-1}}\right]
\times\\
&  \left(  \frac{\overset{m-1}{\underset{\ell=1}{\Pi}}\theta_{\ell}^{\left(
p_{\ell}+1\right)  ^{2}}}{\overset{m-1}{\underset{\ell=1}{\Pi}}\theta_{\ell
}^{p_{\ell}^{2}}}F_{1}+\overset{m-1}{\underset{\kappa=2}{\sum}}\frac
{\overset{\kappa-1}{\underset{k=1}{\Pi}}\theta_{k}^{p_{k}^{2}}\overset
{m-1}{\underset{\ell=\kappa}{\Pi}}\theta_{\ell}^{\left(  p_{\ell}+1\right)
^{2}}}{\overset{m-1}{\underset{\ell=1}{\Pi}}\theta_{\ell}^{p_{\ell}^{2}}%
}F_{\kappa}+F_{m}\right)  \overset{m-1}{\underset{\ell=1}{\Pi}}\theta_{\ell
}^{p_{\ell}^{2}}dx\\
&  =\int_{\Omega}\left[  p_{m}\overset{p_{m}-1}{\underset{p_{m-1}=0}{\sum}%
}...\overset{p_{2}}{\underset{p_{1}=0}{\sum}}C_{p_{m}-1}^{p_{m-1}}...C_{p_{2}%
}^{p_{1}}w_{1}^{p_{1}}w_{2}^{p_{2}-p_{1}}...w_{m}^{p_{m}-1-p_{m-1}}\right]
\times\\
&  \left(  \overset{m-1}{\underset{\ell=1}{\Pi}}\frac{\theta_{\ell}^{\left(
p_{\ell}+1\right)  ^{2}}}{\theta_{\ell}^{p_{\ell}^{2}}}F_{1}+\overset
{m-1}{\underset{\kappa=2}{\sum}}\overset{m-1}{\underset{\ell=\kappa}{\Pi}%
}\frac{\theta_{\ell}^{\left(  p_{\ell}+1\right)  ^{2}}}{\theta_{\ell}%
^{p_{\ell}^{2}}}F_{\kappa}+F_{m}\right)  \overset{m-1}{\underset{\ell=1}{\Pi}%
}\theta_{\ell}^{p_{\ell}^{2}}dx.
\end{align*}
Hence, using condition (\ref{1.11}), we deduce that%
\[
J\leq C_{5}\int_{\Omega}\left[  \overset{p_{m}-1}{\underset{p_{m-1}=0}{\sum}%
}...\overset{p_{2}}{\underset{p_{1}=0}{\sum}}C_{p_{2}}^{p_{1}}...C_{p_{m}%
-1}^{p_{m-1}}w_{1}^{p_{1}}w_{2}^{p_{2}-p_{1}}...w_{m}^{p_{m}-1-p_{m-1}}\left(
1+\overset{m}{\underset{\ell=1}{\sum}}w_{\ell}\right)  \right]  dx.
\]
To prove that the functional $L$ is uniformly bounded on the interval $\left[
0,T^{\ast}\right]  $, let us first write
\begin{align*}
&  \overset{p_{m}-1}{\underset{p_{m-1}=0}{\sum}}...\overset{p_{2}}%
{\underset{p_{1}=0}{\sum}}C_{p_{2}}^{p_{1}}...C_{p_{m}-1}^{p_{m-1}}%
w_{1}^{p_{1}}w_{2}^{p_{2}-p_{1}}...w_{m}^{p_{m}-1-p_{m-1}}\left(
1+\overset{m}{\underset{\ell=1}{\sum}}w_{\ell}\right) \\
&  =R_{p_{m}}\left(  W\right)  +S_{p_{m}-1}\left(  W\right)  ,
\end{align*}
where $R_{p_{m}}\left(  W\right)  $ and $S_{p_{m}-1}\left(  W\right)  $are two
homogeneous polynomials of degrees $p_{m}$ and $p_{m}-1$, respectively. Since
all of the polynomials $H_{p_{m}}$ and $R_{p_{m}}$ are of degree $p_{m}$,
there exists a positive constant $C_{6}$ such that
\begin{equation}
\int_{\Omega}R_{p_{m}}\left(  W\right)  dx\leq C_{6}\int_{\Omega}H_{p_{m}%
}\left(  W\right)  dx. \label{5.21}%
\end{equation}
Applying H\"{o}lder's inequality to the integral $\int_{\Omega}S_{p_{m}%
-1}\left(  W\right)  dx,$ one obtains
\[
\int_{\Omega}S_{p_{m}-1}\left(  W\right)  dx\leq\left(  meas\Omega\right)
^{\frac{1}{p_{m}}}\left(  \int_{\Omega}\left(  S_{p_{m}-1}\left(  W\right)
\right)  ^{\frac{p_{m}}{p_{m}-1}}dx\right)  ^{\frac{p_{m}-1}{p_{m}}}.
\]

Since for all $w_{1},w_{2,...},w_{m-1}\geq0$ and $w_{m}>0,$
\[
\dfrac{\left(  S_{p_{m}-1}\left(  W\right)  \right)  ^{\frac{p_{m}}{p_{m}-1}}%
}{H_{p_{m}}\left(  W\right)  }=\dfrac{\left(  S_{p_{m}-1}\left(  x_{1}%
,x_{2},...,x_{m-1},1\right)  \right)  ^{\frac{p_{m}}{p_{m}-1}}}{H_{p_{m}%
}\left(  x_{1},x_{2},...,x_{m-1},1\right)  },
\]
where $\forall\ell\in\left\{  1,2,...,m-1\right\}  :x_{\ell}=\frac{u_{\ell}%
}{u_{\ell+1}}$ and
\[
\underset{x_{\ell}\rightarrow+\infty}{\lim}\dfrac{\left(  S_{p_{m}-1}\left(
x_{1},x_{2},...,x_{m-1},1\right)  \right)  ^{\frac{p_{m}}{p_{m}-1}}}{H_{p_{m}%
}\left(  x_{1},x_{2},...,x_{m-1},1\right)  }<+\infty,
\]
one asserts that there exists a positive constant $C_{7}$ such that
\begin{equation}
\dfrac{\left(  S_{p_{m}-1}\left(  W\right)  \right)  ^{\frac{p_{m}}{p_{m}-1}}%
}{H_{p_{m}}\left(  W\right)  }\leq C_{7},\text{ for all }w_{1},w_{2}%
,...,w_{m}\geq0. \label{5.22}%
\end{equation}

Hence, the functional $L$ satisfies the differential inequality
\[
L^{\prime}\left(  t\right)  \leq C_{6}L\left(  t\right)  +C_{8}L^{\frac
{p_{m}-1}{p_{m}}}\left(  t\right)  ,
\]
which for $Z=L^{\frac{1}{p_{m}}}$ can be written as
\begin{equation}
p_{m}Z^{\prime}\leq C_{6}Z+C_{8}. \label{5.23}%
\end{equation}
A simple integration gives the uniform bound of the functional $L$ on the
interval $\left[  0,T^{\ast}\right]  $. This ends the proof of the theorem.
\end{proof}

\begin{proof}
[Proof of Corollary \ref{corollary1}]The proof of this corollary is an
immediate consequence of \ref{theorem1} and the inequality%
\begin{equation}
\int_{\Omega}\left(  \overset{m}{\underset{\ell=1}{\sum}}w_{\ell}\left(
t,x\right)  \right)  ^{p}dx\leq C_{9}L\left(  t\right)  \text{ on }\left[
0,T^{\ast}\right]  . \label{5.24}%
\end{equation}
for some $p\geq1.$
\end{proof}

\begin{proof}
[Proof of Proposition \ref{proposition2}]From corollary \ref{corollary1},
there exists a positive constant $C_{10}$ such that
\begin{equation}
\int_{\Omega}\left(  \overset{m}{\underset{\ell=1}{\sum}}w_{\ell}\left(
t,x\right)  +1\right)  ^{p}dx\leq C_{10}\text{ on }\left[  0,T_{\max}\right)
. \label{5.25}%
\end{equation}
From (\ref{1.10}),\ we have%
\begin{align}
\forall\ell &  \in\left\{  1,2,...,m\right\}  :\nonumber\\
\left\vert F_{\ell}\left(  W\right)  \right\vert ^{\frac{p}{N}}  &  \leq
C_{11}\left(  W\right)  \left(  \overset{m}{\underset{\ell=1}{\sum}}W_{\ell
}\left(  t,x\right)  \right)  ^{p}\text{ on }\left[  0,T_{\max}\right)
\times\Omega. \label{5.26}%
\end{align}
Since $w_{1},w_{2},...,w_{m}$ are in $L^{\infty}\left(  0,T^{\ast}%
;L^{p}\left(  \Omega\right)  \right)  $ and $\dfrac{p}{N}>\dfrac{n}{2},$ then
as discussed in the preliminary observations section the solution is global.
\end{proof}

\section{Final Remarks}

Recall that the eigenvectors of the diffusion matrix associated with the
eigenvalue $\overline{\lambda}_{\ell}$ is defined as $\overline{v}_{\ell
}=\left(  \overline{v}_{\ell1},\overline{v}_{\ell2},...,\overline{v}_{\ell
m}\right)  ^{t}$. It is important to note that if $\overline{v}_{\ell}$\ is an
eigenvector then so is $(-1)\overline{v}_{\ell}$. In the region considered in
previous sections, we only used the positive $\overline{v}_{\ell}$. The
remainder of the $2^{m}$ regions can be formed using negative versions of the
eigenvectors. In each region, the reaction-diffusion system with a
diagonalized diffusion matrix is formed by multiplying each of the $m$
equations in (\ref{1.1}) by the corresponding element of either $\overline
{v}_{\ell}$ or $(-1)\overline{v}_{\ell}$ and then adding the $m$ equations
together. The equations multiplied by elements of $\overline{v}_{\ell}$ form a
set $L$, whereas the equations multiplied by elements of $(-1)\overline
{v}_{\ell}$ form a set $Z$. Hence, we can write define the region in the form:%
\[
\Sigma_{L,Z}=\left\{  \left(  u_{1}^{0},u_{2}^{0},...,u_{m}^{0}\right)  \in%
\mathbb{R}
^{m}\text{ such that }\left\{
\begin{array}
[c]{c}%
w_{\ell}^{0}=\sum_{k=1}^{m}u_{k}^{0}\overline{v}_{\ell k}\geq0,\ell\in L\\
w_{z}^{0}=\left(  -1\right)  \sum_{k=1}^{m}u_{k}^{0}\overline{v}_{zk}%
\geq0,z\in Z
\end{array}
\right.  ,\text{ },\right.
\]
with%
\[
\left\{
\begin{array}
[c]{c}%
\rho_{\ell}^{0}=\sum_{k=1}^{m}\beta_{k}v_{\left(  m+1-\ell\right)  k}%
\geq0,\ell\in L\\
\rho_{\ell}^{0}=\left(  -1\right)  \sum_{k=1}^{m}\beta_{k}v_{\left(
m+1-z\right)  k}\geq0,z\in Z
\end{array}
\right.  .
\]
Using lemma \ref{Lemma0} yields
\[
\Sigma_{L,Z}=\left\{  \left(  u_{1}^{0},u_{2}^{0},...,u_{m}^{0}\right)  \in%
\mathbb{R}
^{m}\text{ such that }\left\{
\begin{array}
[c]{c}%
w_{\ell}^{0}=\sum_{k=1}^{m}u_{k}^{0}\sin\frac{\left(  m+1-\ell\right)  k\pi
}{m+1}\geq0,\ell\in L\\
w_{z}^{0}=\left(  -1\right)  \sum_{k=1}^{m}u_{k}^{0}\sin\frac{\left(
m+1-z\right)  k\pi}{m+1}\geq0,z\in Z
\end{array}
\right.  ,\right.
\]
with
\[
\left\{
\begin{array}
[c]{c}%
\rho_{\ell}^{0}=\sum_{k=1}^{m}\beta_{k}\sin\frac{\left(  m+1-\ell\right)
k\pi}{m+1}\geq0,\ell\in L\\
\rho_{z}^{0}=\left(  -1\right)  \sum_{k=1}^{m}\beta_{k}\sin\frac{\left(
m+1-z\right)  k\pi}{m+1}\geq0,z\in Z
\end{array}
\right.  ,
\]
and
\[
\left\{
\begin{array}
[c]{l}%
L\cap Z=\varnothing\\
L\cup Z=\left\{  1,2,...,m\right\}
\end{array}
\right.  .
\]

\end{document}